\documentclass[11pt,reqno,twoside]{amsart}
\usepackage{graphicx}
\usepackage{amssymb}
\usepackage{epstopdf}
\usepackage[asymmetric,top=2.5cm,bottom=2.8cm,left=2.38cm,right=2.38cm]{geometry}
\usepackage{hyperref}
\usepackage{booktabs} 
\usepackage{array} 
\usepackage{paralist} 
\usepackage{verbatim} 
\usepackage{subfig} 
\usepackage{tabularx}
\usepackage{amsmath,amsfonts,amsthm,mathrsfs,amssymb,cite}
\usepackage[usenames]{color}
\usepackage{bm}
\usepackage{color}

\usepackage{xcolor}
\hypersetup{
	colorlinks,
	linkcolor={blue!80!black},
	urlcolor={blue!80!black},
	citecolor={blue!80!black}
}

\newtheorem{thm}{Theorem}[section]

\newtheorem{lem}{Lemma}[section]

\theoremstyle{definition}
\newtheorem{defn}{Definition}[section]
\theoremstyle{remark}

\numberwithin{equation}{section}

\newcommand{\bmf}[1]{{\mathbf{#1}}}

\title[Inverse elastic obstacle scattering problems by monotonicity method]{Inverse elastic obstacle scattering problems by monotonicity method}

\author{Mengjiao Bai}
\address{School of Mathematics, Jilin University,
	Changchun, Jilin 130012, China.}
\email{baimj24@mails.jlu.edu.cn}

\author{Huaian Diao}
\address{School of Mathematics and Key Laboratory of Symbolic Computation and Knowledge Engineering of Ministry of Education, Jilin University, Changchun, Jilin, China.}
\email{diao@jlu.edu.cn; hadiao@gmail.com}

\author{Weisheng Zhou}
\address{School of Mathematics, Jilin University,
	Changchun, Jilin 130012, China.}
\email{zhouws24@mails.jlu.edu.cn}

\date{} 
\begin{document}
	\maketitle
	
	\begin{abstract}	
	
	We consider the elastic wave scattering problem involving rigid obstacles. This work addresses the inverse problem of reconstructing the position and shape of such obstacles using far-field measurements. A novel monotonicity-based approach is developed for this purpose. By factorizing the far-field operator and utilizing the existence of localized wave functions, we derive a shape characterization criterion for the obstacle boundary. The proposed method employs monotonicity tests to determine the geometric relationship between any given test domain and the actual scatterer. As a result, the shape and location of rigid elastic obstacles can be uniquely identified without requiring any initial guesses or prior knowledge of the physical parameters of the homogeneous background medium.

		\medskip
		
		\noindent{\bf Keywords:}~~ inverse elastic obstacle problem, factorization method, localized wave functions 
		
		\medskip
		
	\end{abstract}

	\section{Introduction}
	
	\subsection{Mathematical setup}

	Let \( D \subset \mathbb{R}^2 \) be a bounded elastic obstacle such that its exterior \( \mathbb{R}^2 \setminus \overline{D} \) is connected. The elastic medium surrounding the obstacle is assumed to be homogeneous, isotropic, and occupies the entire plane \( \mathbb{R}^2 \). The background medium is characterized by the Lamé constants \( \lambda \) and \( \mu \), which satisfy the \emph{strong ellipticity conditions}:
\[
\mu > 0 \quad \text{and} \quad \lambda + \mu > 0.
\]
We assume that the medium has unit (normalized) density. The obstacle \( D \) is illuminated by an incident time-harmonic elastic plane wave \( \bmf{u}^i \), which is a superposition of a compressional (P-) wave and a shear (S-) wave. The incident wave is given by
\begin{align}\label{eq:ui}
	\bmf{u}^i(\mathbf{x}) = a_p \mathbf{d} \, e^{\mathrm{i} k_p \mathbf{d} \cdot \mathbf{x}} + a_s \mathbf{d}^\perp \, e^{\mathrm{i} k_s \mathbf{d} \cdot \mathbf{x}},
\end{align}
where \( \mathbf{d}, \mathbf{d}^\perp \in \mathbb{S}^1 \) are unit vectors such that \( \mathbf{d} \) is the direction of wave propagation, and \( \mathbf{d}^\perp \) is orthogonal to \( \mathbf{d} \). The coefficients \( a_p, a_s \in \mathbb{C} \) are not simultaneously zero, i.e., \( (a_p, a_s) \neq (0, 0) \). The wavenumbers for the compressional and shear waves are defined by
\[
k_p = \frac{\omega}{\sqrt{\lambda + 2\mu}}, \quad k_s = \frac{\omega}{\sqrt{\mu}},
\]
where \( \omega > 0 \) is the angular frequency. It can be directly verified that \( \bmf{u}^i \) satisfies the \emph{two-dimensional Navier equation} in \( \mathbb{R}^2 \):
\begin{equation}
	\mu \Delta \bmf{u}^i + (\lambda + \mu) \nabla \nabla \cdot \bmf{u}^i + \omega^2 \bmf{u}^i = \bmf{0}.
\end{equation}

The interaction of the incident wave \( \bmf{u}^i \) with the obstacle \( D \) gives rise to a scattered field \( \bmf{u} \), which also satisfies the Navier equation in the exterior domain:
\begin{equation}\label{eq:syst1}
	\mu \Delta \bmf{u} + (\lambda + \mu) \nabla \nabla \cdot \bmf{u} + \omega^2 \bmf{u} = \bmf{0}, \quad \text{in } \mathbb{R}^2 \setminus \overline{D}.
\end{equation}
The \emph{total field} is defined as the sum of the incident and scattered fields:
\[
\bmf{u}^{\mathrm{tot}} := \bmf{u}^i + \bmf{u}.
\]
In this paper, we consider the case where the obstacle \( D \) is \emph{impenetrable and rigid}. This implies that the total displacement field vanishes on the boundary of the obstacle:
\begin{equation}\label{eq:BoundaryCondition}
	\bmf{u}^{\mathrm{tot}} = \bmf{0}, \quad \text{on } \partial D.
\end{equation}

	Additionally, the scattered field $\mathbf u$ is required to satisfy the Kupradze radiation condition
	\begin{equation}\label{eq:RadiationCondition}
		\lim_{r\to\infty}r^\frac{1}{2}\left(\frac{\partial\bmf{u}_p}{\partial r}-{\rm i}k_p\bmf{u}_p\right)=0, \quad   	\lim_{r\to\infty}r^\frac{1}{2}\left(\frac{\partial\bmf{u}_s}{\partial r}-{\rm i}k_s\bmf{u}_s\right)=0, \ r:=|\bmf{x}|,
	\end{equation}
	where 
	$$
	\bmf{u}_p=-\frac{1}{k_p^2}\nabla \nabla \cdot \mathbf u,\quad \bmf{u}_s=-\frac{1}{k_s^2}\mathbf{curl}curl\mathbf u,
	$$
	are known as the compressional and shear wave components of $\mathbf u$ respectively.
	Given a vector function $\mathbf u:= \left( u_1, u_2\right)^{\top} $ and a scalar function $u$, the two-dimensional $\mathbf{curl}$ and curl operators is defined as follows
	\begin{equation}\label{eq:curl}
		{\rm curl}~ \bm{u}=\partial_1u_2-\partial_2u_1, \quad {\bf curl}~{u} =(\partial_1 u,-\partial_2 u)^\top.
	\end{equation}
	Then, we define the Helmholtz decomposition of the scattered wave $\bmf{u}$ in \eqref{eq:BoundaryCondition} as follows
	$$
	\bmf {u}=\bmf{u}_p+\bmf{u}_s.
	$$ 
	
	It is well-known that the scattered field satisfying the Kupradze radiation condition has the following asymptotic behavior
	\begin{equation}\label{eq:asymptotic}
	\bmf{u}(\bmf{x})= \frac{{e}^{\mathrm{i} k_{p} |\mathbf x|}}{\sqrt{|\mathbf x|}} \mathbf u_{p}^{\infty}(\hat{\bmf{x} } )+\frac{{e}^{\mathrm{i} k_{s} |\mathbf x|}}{\sqrt{|\mathbf x|} } \mathbf u_{s}^{\infty}(\hat{\bmf{x} }) +O\left(\frac{1}{|\mathbf x|^{3 / 2}}\right) ,\quad |\mathbf x|\to \infty,
    \end{equation}
	where $\mathbf u_{p}^{\infty}$ and $\mathbf u_{s}^{\infty}$ defined on the unit sphere $\mathbb {S}:=\left\lbrace \hat{\mathbf x} \in \mathbb R^2 \big | |\hat{\mathbf x} |=1\right\rbrace$ are vector-valued analytic functions and referred to the far-field pattern of $\mathbf u_{p}$ and $\mathbf u_{s}$ respectively. We typically consider ($\mathbf u_{p}^{\infty}, \mathbf u_{s}^{\infty}$) as an element in $\left[ L^2(\mathbb S)\right] ^2$ space, and denote by $\mathbf {u}^{\infty}\left( \cdot, \mathbf d, a_p,a_s\right)$ the far-field pattern associated with the scattered field produced by incident waves of the form $\eqref{eq:ui}$. In this work, our goal is to reconstruct the shape of the rigid scatterer from the far-field pattern $\mathbf {u}^{\infty}\left( \cdot, \mathbf d, a_p,a_s\right)$.


\subsection{Connections to previous results and main findings}

Inverse scattering problems have been widely applied in various fields such as oil exploration and medical imaging. Over the past three decades, numerical methods for solving inverse scattering problems have undergone rapid development. Addressing inverse scattering problems typically requires overcoming their inherent nonlinearity and ill-posedness. To this end, iterative methods were initially proposed \cite{Roge81}. However, these methods typically require prior knowledge of certain a priori information and necessitate solving a forward problem at each iteration step, which significantly increases computational costs. To enhance computational efficiency, non-iterative methods have been introduced, among which sampling-type methods have garnered considerable attention due to their computational efficiency and ease of implementation. Examples of such methods include the direct sampling method \cite{Aren01}, linear sampling method \cite{CK96}, factorization method \cite{HKS12} and monotonicity method \cite{TR02,BU13}. The factorization method, first introduced by  Arens \cite{Aren01} for the theoretical analysis of two-dimensional rigid obstacle scattering. Additionally, the factorization method has been extended to penetrable scatterers \cite{CKAKGDK06} and crack detection \cite{JQG18}.

Monotonicity methods, as a distinct category of non-iterative techniques, was originally introduced for solving electrical impedance tomography (EIT) problems \cite{TR02}, where the authors developed a non-iterative inversion technique based on the monotonicity of the resistance matrix. Later, Harrach and Ullrich proposed a shape reconstruction method for EIT by leveraging the monotonicity of the Neumann-to-Dirichlet (NtD) operator with respect to conductivity \cite{BU13}. In recent years, researchers further applied the monotonicity method to the reconstruction of elastic inclusions, proposing normalized and linearized monotonicity methods to recover the inhomogeneities in Lam\'e parameters and density relative to the background medium using the monotonicity of the NtD operator\cite{EH21,EH23,EH24}. It is noteworthy that in these monotonicity-based methods for elastic inclusion reconstruction, authors typically rely on the monotonicity of the NtD operator to reconstruct inhomogeneities within the inclusions. Meanwhile, we note that references \cite{AG20,AG23} proposed an innovative scheme based on the monotonicity method, achieving position reconstruction of acoustic and electromagnetic wave scatterers using far-field data. The implementation of this monotonicity method relies on both the advancement of the factorization method and the existence of localized wave functions—specifically, the existence of a sequence of functions that can attain arbitrarily large norms in certain designated regions while exhibiting arbitrarily small norms in other designated regions.

In this study, we aim to reconstruct the shape of a rigid obstacle from far-field measurements using a monotonicity-based approach. By leveraging the factorization of the far-field operator and the existence of localized wave functions, we establish a shape characterization theorem for elastic obstacles. Specifically, the reconstruction is achieved by computing the number of positive (negative) eigenvalues of a linear combination of the far-field operator and a probe operator. This allows us to determine whether a test domain lies inside the obstacle, ultimately leading to a complete shape reconstruction.

Finally, the remainder of the paper is structured as follows. In Section \ref{sec:far-field Operator decomposition}, we present a factorization result for the far-field operator associated with rigid elastic scatterers and define key operators and notations used in subsequent proofs. Section \ref{sec:LocalizedWaveFunction} is devoted to proving the existence of localized wave functions, i.e., constructing a function sequence that diverges within a specified region while vanishing elsewhere. In Section \ref{sec:Characterization of the obstacle}, we establish a shape characterization theorem for the reconstruction of rigid elastic scatterers, demonstrating that the shape of an impenetrable obstacle can be approximately reconstructed by computing the numbers of positive and negative eigenvalues of a specific operator combination.

	\section{Decomposition of The Far-field Operator}\label{sec:far-field Operator decomposition}
	This section presents the key properties of the single-layer potential operator, the factorization theorem for the far-field operator associated with rigid elastic scatterers, and an estimate for the Data-to-pattern operator. Although the properties of the single-layer potential operator and the factorization theorem of the far-field operator can be found in \cite{Aren01}, we restate these known results here to ensure clarity in subsequent proofs. Furthermore, to elucidate the connections between the Herglotz operator, the single-layer potential operator, and the Data-to-pattern operator, we provide a detailed proof of the far-field operator factorization theorem. Finally, we establish an estimate for the Data-to-pattern operator, which lays the theoretical foundation for further analysis.
	
	Before introducing the far-field operator, we will first present several Hilbert spaces that may be utilized in our discussion. For a bounded domain $D$ of class $C^2$, let $\left[L^2\left(\partial D\right) \right] ^2$ and $\left[H^{1/2} \left(\partial D\right) \right] ^2$ denote the usual Sobolev spaces of two-dimensional vector fields defined on $\partial D$. Let $\mathbf{g}=\left( g_p,g_s\right)$, $\mathbf{h}=\left( h_p,h_s\right)$, $\mathbf{g},\mathbf{h} \in \left[L^2\left(\mathbb S\right) \right] ^2$, the inner product on the Hilbert space $\left[L^2\left(\mathbb S\right) \right] ^2$  is defined as
	\begin{align}\label{eq:InnerProduct}
		\left\langle \mathbf{g},\mathbf{h}\right\rangle := \frac{\omega}{k_p}\int_{\partial D} g_p(\mathbf d) \overline{h_p(\mathbf d)}ds(\mathbf d)+\frac{\omega}{k_s}\int_{\partial D} g_s(\mathbf d) \overline{h_s(\mathbf d)}ds(\mathbf d).
	\end{align}
	We usually denote the dual space of  $\left[H^{1/2}\left(\partial D\right) \right] ^2$ as  $\left[H^{-1/2}\left(\partial D\right) \right] ^2$ with respect to inner product in $\left[L^2\left(\partial D\right) \right] ^2$ and their dual pairing is denoted as $\left( \cdot,\cdot\right) $.
	
	For a given $\mathbf{g}=\left( g_p, g_s\right)  \in \left[L^2\left(\mathbb S\right) \right] ^2$ , the elastic Herglotz wavefunction with density $\mathbf{g}$ is defined as
	\begin{align}\label{eq:HorglotzWave}
		\mathbf{v_g}(\mathbf{x}):=e^{-i\pi /4} \int_{\mathbb S}  \left\lbrace \sqrt{\frac{k_p}{\omega}}\mathbf{d} e^{ik_p\mathbf{d}\cdot\mathbf{x}}g_p(\mathbf{d})+\sqrt{\frac{k_s}{\omega}}\mathbf{d}^{\perp} e^{ik_s\mathbf{d}\cdot\mathbf{x}}g_s(\mathbf{d})\right\rbrace ds(\mathbf{d}), \qquad \mathbf{x}\in\mathbb{R}^2,
	\end{align}
   where it is evident that the Herglotz wavefunction can be regarded as a superposition of plane waves of the form $\eqref{eq:ui}$. With this definition, we can now proceed to introduce the far-field operator $\mathbf{F}$, which maps the incident Herglotz wave density $\mathbf{g}$ to the far-field pattern of the scattered wave.
		
	For $\hat{\mathbf x}\in \mathbb{S}$, the far-field operator $\mathbf{F}$: $\left[L^2\left(\mathbb S\right) \right] ^2 \to \left[L^2\left(\mathbb S\right) \right] ^2$ is defined as
	\begin{equation}\begin{aligned}\label{eq:FarfieldOperator}
		\mathbf{Fg}(\hat{\mathbf x})
		 & :=e^{-i\pi /4} \int_{\mathbb S} \mathbf{u}^{\infty} \left( \hat{\mathbf x}, \mathbf d, \sqrt{\frac{k_p}{\omega}}g_p(\mathbf d),\sqrt{\frac{k_s}{\omega}}g_s(\mathbf d) \right) ds(\mathbf d)\\
		& =e^{-i\pi /4} \int_{\mathbb S}\left\lbrace \sqrt{\frac{k_p}{\omega}}\mathbf{u}^\infty (\hat{\mathbf x},\mathbf d,1,0)g_p(\mathbf d)+\sqrt{\frac{k_s}{\omega}}\mathbf{u}^\infty (\hat{\mathbf x},\mathbf d,0,1)g_s(\mathbf d)\right\rbrace ds(\mathbf d).
	\end{aligned}\end{equation}
    Here, $\mathbf{u}^\infty(\hat{\mathbf{x}},\mathbf{d},1,0)$ and $\mathbf{u}^\infty(\hat{\mathbf{x}},\mathbf{d},0,1)$ represent the far-field patterns generated by purely compressional  and purely shear incident waves, respectively.
	
	Before formally presenting the factorization of the far-field operator, we need to introduce the single-layer potential operator $\mathbf S$, the Herglotz wave operator $\mathbf H$ and the data-to-pattern operator $\mathbf G$. Here, we first provide the definition of the single-layer potential operator. We define the single-layer potential operator $\mathbf{S}:\left[H^{-1/2}\left(\partial D\right) \right] ^2 \to \left[H^{1/2}\left(\partial D\right) \right] ^2$ on $\partial D$ as follows
	\begin{align}\label{eq:SingleLayerOperater}
		\mathbf S \phi \left( \mathbf x\right) :=\int_{\partial D} \mathbb G \left( \mathbf x, \mathbf y\right) \phi\left(\mathbf y\right) ds\left( \mathbf y\right) , \qquad \mathbf x\in \partial D,
	\end{align}
	where $\mathbb G \left( \mathbf x, \mathbf y\right)$ is the Green's tensor of the Navier equation, i.e.,
	$$
	\mathbb G \left( \mathbf x, \mathbf y\right):=\frac{i}{4\mu} H_0^{\left( 1\right) }\left( k_s|\mathbf x -\mathbf y|\right) \mathbf I_2 +\frac{i}{4\omega^2} \nabla_x^{\top} \nabla_x \left( H_0^{\left( 1\right) }\left( k_s|\mathbf x -\mathbf y|\right)-H_0^{\left( 1\right) }\left( k_p|\mathbf x -\mathbf y|\right)\right) .
	$$
	Here, $\mathbf I_2$ is the $2 \times 2$ identity matrix, $H_0^{\left( 1\right) }$ is the Hankel function of the first kind and of order 0.
	
	The following properties of the single-layer potential operator $\mathbf S$, which have been established in \cite{Aren01}, are presented below. Here, $\mathbf{S}_i$ denotes the far-field operator corresponding to the angular frequency $\omega = i$.
	 \begin{lem}\label{SProperty}
	 	Assume $\omega^2$ is not a Dirichlet eigenvalue of $-\triangle^*$ in $D$,
	 	\begin{itemize}
	 		\item [(1).] $\mathbf S$ is an isomorphism from the Sobolev space $\left[H^{-1/2}\left(\partial D\right) \right] ^2 $ onto $\left[H^{1/2}\left(\partial D\right) \right] ^2$.
	 		\item [(2).] $Im\left( \phi,\mathbf S\phi\right) =0$ for some $\phi \in \left[H^{1/2}\left(\partial D\right) \right] ^2$ implies $\phi=0$.
	 		\item [(3).] The operator $\mathbf S_i$ is compact, self adjoint, and positive defined  in $\left[L^2\left(\partial D\right) \right] ^2 $. Moreover, $\mathbf S_i$ is coercive as an operator from $\left[H^{-1/2}\left(\partial D\right) \right] ^2 $ onto $\left[H^{1/2}\left(\partial D\right) \right] ^2$, i.e. there exists $c_0>0$ such that
	 		$$
	 		\left( \phi,\mathbf S_i \phi\right) \geq c_0\Vert \phi\Vert^2_{\left[H^{-1/2}\left(\partial D\right) \right] ^2}, \qquad \phi \in \left[H^{-1/2}\left(\partial D\right) \right] ^2.
	 		$$
	 		Furthermore, there exists a self adjoint and positive definite square root $\mathbf S_{i}^{1/2}$ of $\mathbf S_i$ and $\mathbf S_{i}^{1/2}$ is an isomorphism from $\left[H^{-1/2}\left(\partial D\right) \right] ^2 $ onto $\left[L^2\left(\partial D\right) \right] ^2 $ and from $\left[L^2\left(\partial D\right) \right] ^2 $ onto $\left[H^{1/2}\left(\partial D\right) \right] ^2 $.
	 		\item [(4).] The difference $\mathbf S - \mathbf S_i$ is compact from $\left[H^{-1/2}\left(\partial D\right) \right] ^2 $ to $\left[H^{1/2}\left(\partial D\right) \right] ^2$.
	 	\end{itemize}
	 \end{lem}
	 	 Next, we introduce the Herglotz wave operator $\mathbf H: \left[L^2\left(\mathbb S\right) \right] ^2 \to \left[H^{1/2}\left(\partial D\right) \right] ^2$
	 	\begin{align}\label{eq:HerglotzOperatorD}
	 		\mathbf{Hg}\left( \mathbf x\right) :=e^{-i\pi /4} \int_{\mathbb S}  \left\lbrace \sqrt{\frac{k_p}{\omega}}\mathbf{d} e^{ik_p\mathbf{d}\cdot\mathbf{x}}g_p(\mathbf{d})+\sqrt{\frac{k_s}{\omega}}\mathbf{d}^{\perp} e^{ik_s\mathbf{d}\cdot\mathbf{x}}g_s(\mathbf{d})\right\rbrace ds(\mathbf{d}).
	 	\end{align}
	 	It is straightforward to obtain the adjoint operator of $\mathbf H$ have the following form
	 	\begin{align}\label{eq:HerglotzOperatorDA}
	 		\mathbf{H^* \phi }\left( \mathbf d\right) :=e^{i\pi /4} \int_{\partial \mathbb D}  \left\lbrace \sqrt{\frac{\omega}{k_p}}\mathbf{d} \cdot e^{-ik_p\mathbf{d}\cdot\mathbf{x}}\mathbf{\phi}\left( \mathbf{x}\right)  , \sqrt{\frac{k_s}{\omega}}\mathbf{d}^{\perp} \cdot e^{-ik_s\mathbf{d}\cdot\mathbf{x}}\mathbf{\phi}\left( \mathbf{x}\right)\right\rbrace ds(\mathbf{x}).
	 	\end{align}
	 	
	 	Finally, we introduce the data-to-pattern operator $\mathbf G$. Since the exterior Dirichlet boundary value problem of the Navier equation has a unique radiating solution, the mapping from the boundary values to the far-field is well-defined. We define this map as data-to-pattern operator $\mathbf G : \left[H^{1/2}\left(\partial D\right) \right] ^2 \to \left[L^2\left(\mathbb S\right) \right] ^2 $ by
	 	\begin{align}\label{eq:DatatoPatternOperator}
	 		\mathbf G \mathbf f:=\mathbf{u}^{\infty},
	 	\end{align}
	 	where $\mathbf{u}^{\infty}$is the far-field pattern of the solution to the exterior Dirichlet boundary value problem with boundary values $\mathbf f$. Using the existence and uniqueness of solutions to the exterior Dirichlet boundary value problem for elastic waves \cite{HH93}, we conclude that the data-to-pattern operator $\mathbf{G}: \left[H^{1/2}\left(\partial D\right) \right] ^2 \to \left[L^2\left(\mathbb S\right) \right] ^2$ is injective.
	 	
	 	Now, from \cite{Aren01} and \cite{Kress96}, we obtain the factorization of the far-field operator $\mathbf F$ and establish the relationships between the operators $\mathbf S$, $\mathbf H$ and $\mathbf G$. We present these known results in the following lemma.
	\begin{lem}\label{Fdecomposition}
		For the far-field operator $\mathbf F$,  Herglotz wave operator $\mathbf H$ and data-to-pattern operator $\mathbf G$, we have the following two results
		\begin{itemize}
		\item [(1).] The far-field operator $\mathbf F$ can be represented as
	  \begin{align}\label{eq:FarfieldOperatorDecomposition}
		\mathbf F=-\sqrt{8\pi\omega}\mathbf{G S^*G^*},
	  \end{align}
		where $\mathbf{G^*}:\left[L^2\left(\mathbb S\right) \right] ^2 \to \left[H^{-1/2}\left(\partial D\right) \right] ^2$ and $\mathbf{S^*}:\left[H^{-1/2}\left(\partial D\right) \right] ^2 \to \left[H^{1/2}\left(\partial D\right) \right] ^2$ denote the adjoints of $\mathbf G$ and $\mathbf S$, respectively. 
		\item [(2).] The operators $\mathbf S$, $\mathbf H$ and $\mathbf G$ and their adjoints satisfy the following two equations
		\begin{align}\label{eq:HGSoperator}
			\mathbf{H^*}=\sqrt{8\pi\omega}\mathbf{G S}\quad, \quad \mathbf{H}=\sqrt{8\pi\omega}\mathbf{S^* G^*}.
		\end{align}
    	\end{itemize}
	\end{lem}
	
	In the following lemma, we establish an estimate for the interaction between the operator $\mathbf{G}$ and an arbitrary compact operator $\mathbf K$, which will be essential for deriving the shape characterization through the monotonicity method.
	\begin{lem}\label{DatatoPatternProperty}
	Let $\mathbf K:\left[H^{-1/2}\left(\partial D\right) \right] ^2 \to \left[H^{1/2}\left(\partial D\right) \right] ^2$ is a compact and self and adjoint operator, then for any constant $c>0$, there exists a finite-dimensional subspace $V \subseteq \left[L^2\left(\mathbb S\right) \right] ^2$, such that
	\begin{align}\label{eq:DatatoPatternProperty}
		\left|\left( \mathbf G ^* \mathbf {\phi} ,\mathbf {KG}^* \mathbf {\phi}\right) \right| \leq c\Vert \mathbf G ^* \mathbf \phi\Vert _{\left[H^{-1/2}\left(\partial D\right) \right] ^2} ^2, \qquad \forall \mathbf {\phi} \in V^\perp.
	\end{align}
	Here $\left( \cdot,\cdot\right) $denotes the duality pairing in $\left( \left[H^{-1/2}\left(\partial D\right) \right] ^2 ,\left[H^{1/2}\left(\partial D\right) \right] ^2\right) $.
	\end{lem}

	\begin{proof}
	The eigenfunctions of a self-adjoint compact operator on $\left[L^2(\partial D)\right]^2$ form a complete orthogonal basis, this property that does not hold in Sobolev spaces $\left[H^{-1/2}\left(\partial D\right) \right] ^2$. We therefore use $\mathbf{S}_i^{1/2}$ to modify the operator domains.
	
	From lemma \ref{SProperty}, we can deduce that $\mathbf S_{i}^{1/2}$ is an isomorphism from $\left[H^{-1/2}\left(\partial D\right) \right] ^2 \to \left[L^2\left(\partial D\right) \right] ^2 ,$  thus it inverse operator $\mathbf S_{i}^{-1/2}:\left[L^2\left(\partial D\right) \right] ^2 \to \left[H^{-1/2}\left(\partial D\right) \right] ^2$  exists, and then $\mathbf S_{i}^{-1/2} \mathbf S_{i}^{1/2}$ is the identity operator from $\left[H^{-1/2}\left(\partial D\right) \right] ^2$ to $\left[H^{-1/2}\left(\partial D\right) \right] ^2$. Since $\mathbf S_{i}^{1/2}$ is a self adjoint operator, we can derived that $\mathbf S_{i}^{1/2} \mathbf S_{i}^{-1/2}$ is the identity operator from $\left[H^{1/2}\left(\partial D\right) \right] ^2$ to $\left[H^{1/2}\left(\partial D\right) \right] ^2$. Then we have
\begin{equation}\begin{aligned} \label{eq:inequality1}
	\left|\left( \mathbf G ^* \mathbf {\phi} , \mathbf {KG}^* \mathbf {\phi}\right) \right|
	 & =\left|\left( \mathbf G ^* \mathbf {\phi}, \mathbf S_{i}^{1/2} \mathbf S_{i}^{-1/2} \mathbf K \mathbf S_{i}^{-1/2} \mathbf S_{i}^{1/2} \mathbf {G}^* \mathbf {\phi}\right)\right|\\
		& =\left|\left( \mathbf S_{i}^{1/2} \mathbf {G}^* \mathbf {\phi}, \mathbf {\tilde K}\mathbf S_{i}^{1/2} \mathbf {G}^* \mathbf {\phi} \right)\right|,
	\end{aligned}\end{equation}
	where $ \mathbf {\tilde K}:=\mathbf S_{i}^{-1/2} \mathbf K \mathbf S_{i}^{-1/2}$ is a compact and self adjoint operator from $\left[L^2\left(\partial D\right) \right] ^2$ to $\left[L^2\left(\partial D\right) \right] ^2$. From the spectral theory of compact operators, we can know that  $\mathbf {\tilde K}$ has a countable number of eigenvalues, and these eigenvalues do not have a non-zero accumulation point. Since $ \mathbf {\tilde K}$ is a compact and self adjoint  operator defined on $\left[L^2\left(\partial D\right) \right] ^2$, its eigenvalues are all real numbers, and the eigenvectors corresponding to these eigenvalues form a complete orthogonal basis of $\left[L^2\left(\partial D\right) \right] ^2$. Now, we can utilize the eigenvectors of the operator $\mathbf {\tilde K}$ to find the desired subspace $V^\perp$. We denote the space $\tilde V$ which are formed by the eigenvectors corresponding to the eigenvalues of $\mathbf {\tilde K}$ that are greater than $\tilde c$, here $\tilde c:=c/{||\mathbf S^{1/2}_i||^2_{\left[H^{-1/2}\left(\partial D\right) \right] ^2 \to \left[L^2\left(\partial D\right) \right] ^2}}$, thus $\tilde V$ is finite-dimensional, and for any $\tilde v \in {\tilde V}^{\perp} $, we have the following inequality
	 \begin{align}\label{eq:inequality2}
	 	\left|\left( \tilde v,\mathbf {\tilde K} \tilde v \right) \right| \leq \tilde c \Vert \tilde v\Vert^2_{\left[L^2\left(\partial D\right) \right] ^2}.
	 \end{align}
	Let $\phi \in \left[L^2\left(\partial D\right) \right] ^2$, then from the definition of the orthogonal complement space, we know that $\mathbf S^{1/2}_i \mathbf G ^* \phi \in {\tilde V}^{\perp}$ if and only if the following equality holds
	$$
	0=\left( \mathbf S^{1/2}_i \mathbf G ^* \phi,\tilde v \right)=\left( \phi, \mathbf G \mathbf S^{1/2}_i \tilde v \right) ,\qquad \tilde v \in \tilde V.
	$$
	Therefore, $\mathbf S^{1/2}_i \mathbf G ^* \phi \in {\tilde V}^{\perp}$ if and only if $\phi \in {\mathbf G \mathbf S^{1/2}_i \tilde V}^{\perp}$, we can define $V:=\mathbf G \mathbf S^{1/2}_i \tilde V \subseteq \left[L^2\left(\mathbb S\right) \right] ^2$, then we have
	$$
	\rm {dim}\left( V \right) =\rm {dim}\left( \mathbf G \mathbf S^{1/2}_i \tilde V\right) \leq \rm {dim}\left( \tilde V\right) < \infty.
	$$
	From  \eqref{eq:inequality1} and \eqref{eq:inequality2} , for any $\phi \in V^{\perp}$ we have
\begin{equation}\begin{aligned}
	\left|\left( \mathbf G ^* \mathbf {\phi} , \mathbf {KG}^* \mathbf {\phi}\right)\right| 
	&\leq \tilde c \left\|\mathbf S^{1/2}_i \mathbf G ^* \phi\right\|^2_{\left[L^2\left(\partial D\right) \right] ^2}\\
	&\leq \tilde c \left\|\mathbf S^{1/2}_i\right\|^2_{ \left[H^{-1/2}\left(\partial D\right) \right] ^2 \to \left[L^{2}\left(\partial D\right) \right] ^2} \left\|\mathbf G ^* \phi\right\|^2_{\left[H^{-1/2}\left(\partial D\right) \right] ^2}\\
	&\leq c\left\|\mathbf G ^* \phi\right\|^2_{\left[H^{-1/2}\left(\partial D\right) \right] ^2}.
	\end{aligned}\end{equation}

	The proof is complete. 
	\end{proof}

\section{The Existence of Localized Wave Function}\label{sec:LocalizedWaveFunction}
	In this section, we primarily discuss the existence of the localized wave function, that is, we aim to prove the existence of a sequence of functions whose norm tend to infinity in a certain region and approach zero in another given region. The existence of localized wave functions plays an important role in demonstrating that a test region does not contained in obstacle. Before formally commencing the content of this section, we first introduce the following operators, which will facilitate the discussion of our topic in this section.
	
	Similar to the definition of $\eqref{eq:HerglotzOperatorD}$, we define the Herglotz wave operator $\mathbf H_B$ with its range in $\left[H^{1/2}\left(\partial B\right) \right] ^2$ are as follows
	\begin{align}\label{eq:HerglotzOperatorB}
		\mathbf H_B \mathbf g \left( \mathbf x\right) :=e^{-i\pi /4} \int_{\mathbb S}  \left\lbrace \sqrt{\frac{k_p}{\omega}}\mathbf{d} e^{ik_p\mathbf{d}\cdot\mathbf{x}}g_p(\mathbf{d})+\sqrt{\frac{k_s}{\omega}}\mathbf{d}^{\perp} e^{ik_s\mathbf{d}\cdot\mathbf{x}}g_s(\mathbf{d})\right\rbrace ds(\mathbf{d}).
	\end{align}
	We observe that the operators $\mathbf H_B$ and $\mathbf H$ differ only in their ranges.
	
   Next, we define the restriction operator $\mathbf R_{\tau}$ and its adjoint operator $\mathbf R^*_{\tau}$. Let $\tau \subseteq \partial B$ is relatively open, we define the restriction operator $\mathbf R_{\tau}:\left[H^{1/2}\left(\partial B\right) \right] ^2 \to \left[H^{1/2}\left(\tau \right) \right] ^2 $ are as follows
	\begin{align}\label{eq:RestrictionOperator}
	 \mathbf R_{\tau} \mathbf f:= \mathbf f|_{\tau}.
	\end{align}
	In order to define the adjoint of the restriction operator, we first introduce the following Sobolev space
	\begin{align}
		\left[H^{-1/2}_{supp}\left(\tau \right) \right] ^2:=\left\lbrace \mathbf f \in \left[H^{-1/2}\left(\partial B\right) \right] ^2 | supp \mathbf f \subseteq \overline{\tau}\right\rbrace .
	\end{align}
	Now, we can define it is adjoint operator $\mathbf R^*_{\tau}: \left[H^{-1/2}_{supp}\left(\tau \right) \right] ^2 \to \left[H^{-1/2}\left(\partial B\right) \right] ^2 $ as
	\begin{equation}\label{eq:RestrictionOperatorA}
		\mathbf R^*_{\tau}:=
		\begin{cases}
			\mathbf f,& \mbox{on}\quad \tau,\\
			\bmf 0, &  \mbox{on}\quad \partial B \backslash \tau .
		\end{cases}
	\end{equation}
	
	Utilizing the  Herglotz wave operator $\mathbf H_B$ and the restriction operator $\mathbf R_{\tau}$ defined above, we can define the operator $\mathbf H_{\tau}:= \mathbf R_{\tau} \mathbf H_B$, therefore, using equation $\eqref{eq:HGSoperator}$, we obtain the following result by replacing $D$ with $B$
	\begin{align}\label{eq:RestriH}
		\mathbf H^*_{\tau}=\mathbf H^*_B \mathbf R^*_{\tau}=\sqrt{8\pi\omega}\mathbf{G_B S_B}  \mathbf R^*_{\tau}.
	\end{align}
	Now, we illustrate through the following theorem that the intersection of the ranges of operators $\mathbf H^*_{\tau}$ and $\mathbf G$ can only be the zero element.
	\begin{thm}\label{RangeHG}
		Let $B, D \subset \mathbb R^2$ are open and Lipschitz bounded domain, if $B \nsubseteq D$ ,let $\tau \subseteq \partial B \setminus \overline{D} $ is relatively open and $\mathbb R^2 \setminus \left( \overline{\tau \cup D }\right) $ is connected, then we have
			\begin{align}
				R\left( \mathbf H^*_{\tau} \right) \cap R\left( \mathbf G\right) =\left\lbrace 0\right\rbrace .
			\end{align}
	\end{thm}
	
	\begin{proof}
	Let $\mathbf h \in 	R\left( \mathbf H^*_{\tau} \right) \cap R\left( \mathbf G\right)$, since $\mathbf h$ lies in the ranges of $\mathbf H^*_{\tau}$ and $\mathbf G$, we can deduce that there exists $\mathbf{\phi} _{\tau} \in \left[H^{-1/2} \left(\tau \right) \right] ^2$, $\mathbf f \in \left[H^{1/2}\left(\partial D\right) \right] ^2$, such that
	$$
	\mathbf h= \mathbf H^*_{\tau} \phi _{\tau}= \mathbf G \mathbf f.
	$$
	Since single-layer potentials with density $\phi _{\tau}$ is the solution to the Navier equation in $\mathbb R^2 \backslash \overline{\tau}$ and satisfy the Kupradze radiation condition, it follows that $\sqrt{8\pi\omega} \mathbf S_B \mathbf R^*_{\tau} \phi _{\tau}$ is also a solution to the  the Navier equation in $\mathbb R^2 \backslash \overline{\tau}$ and satisfy the Kupradze radiation condition, let $\mathbf v _1:= \sqrt{8\pi\omega} \mathbf S_B \mathbf R^*_{\tau} \phi _{\tau} \in \left[ H^1_{loc} \left( \mathbb R^2 \backslash \overline{\tau} \right) \right] ^2$, then $\mathbf v _1$ is the radiating solution to the following equation
	\begin{equation}
			\triangle ^* \mathbf v _1 + \omega ^2 \mathbf v _1 =0,\qquad \mbox{in}\quad \mathbb R^2 \backslash \overline{\tau}.
	\end{equation}
	From \eqref{eq:HGSoperator}, we can infer that $\mathbf H^*_{\tau} \phi _{\tau}$ is $\sqrt{8\pi\omega}$ times the far-field pattern of the single-layer potential $\mathbf S_B$ with density $\mathbf R^*_{\tau} \phi _{\tau}$. Then, we can derive that $\mathbf H^*_{\tau} \phi _{\tau}= \mathbf v^{\infty} _1$. 
	
	On the other hand, if we let $\mathbf v _2 \in \left[ H^1_{loc} \left( \mathbb R^2 \backslash \overline{D} \right) \right] ^2$ be the radiating solution to the Navier equation satisfying $ \mathbf v _2= \mathbf f$ on $\partial D $, then according to the definition of the data-to-pattern operator, we can obtain $ \mathbf G \mathbf f= \mathbf v^{\infty} _2$. Clearly, $\mathbf v _2 $ is the radiating solution to the following equation
   \begin{equation}
	\triangle ^* \mathbf v _2 + \omega ^2 \mathbf v _2 =0,\qquad \mbox{in}\quad \mathbb R^2 \backslash \overline{D}.
   \end{equation}
	Since $\mathbf v^{\infty} _1=\mathbf H^*_{\tau} \phi _{\tau}= \mathbf h= \mathbf G \mathbf f= \mathbf v^{\infty} _2$, according to the Rellich lemma for elastic waves, we can derive that
	$$
	\mathbf v _1=\mathbf v _2, \qquad \mbox{in}\quad \mathbb R^2 \backslash \overline{\tau \cup D}.
	$$
	Then we can define $\mathbf v \in \left[ H^1_{loc} \left( \mathbb R^2 \right) \right] ^2$
	\begin{equation}
		\mathbf v:=
		\begin{cases}
			\mathbf v _1=\mathbf v _2,& \mbox{in}\quad \mathbb R^2 \backslash \overline{\tau \cup D},\\
			\mathbf v _1, &  \mbox{in}\quad D ,	\\
			\mathbf v _2, &  \mbox{on}\quad \tau .
		\end{cases}
	\end{equation}
	Then $\mathbf v$ is an entire solution to the Navier equation, so we have
	$$
	\mathbf h= \mathbf v^{\infty} _1= \mathbf v^{\infty} _2 =0.
	$$
	
	The proof is complete.	
	\end{proof}

To prove the existence of localized wave functions, we also need to introduce the following two lemmas, the proofs of which can be directly found in \cite{AG20}.
\begin{lem}\label{Inquallity}
	Assume $X, Y, Z$ be Hilbert spaces, let $A_1:X \to Y$ and $A_2: X \to Z$ are linear operators, then the following two statements are equivalent
	\begin{itemize}
		\item [(1).] There exists a constant $C>0$ such that $||A_1 x||_Y \leq C ||A_2 x||_Z, \quad \forall x \in X$.
		\item [(2).] $R \left( A^*_1 \right) \subseteq R \left( A^*_2 \right) $.
	\end{itemize}
\end{lem}

\begin{lem}\label{Dim}
	Assume $V, Z_1, Z_2 $ be subspaces of a vector space $Z$, if 
	$$
	Z_1 \cap Z_2= \left\lbrace 0 \right\rbrace \quad and \quad Z_1 \subseteq Z_2 +V,
	$$
	then $\rm {dim}\left(  Z_1 \right) \leqslant \rm {dim}\left(  V \right) $.
\end{lem}
Following the above introduction, we will demonstrate the existence of localized wave functions through the following theorem. 
\begin{thm}\label{LocialWaveFounc}
 Let $B, D  \subseteq \mathbb R^2$ is open and Lipschitz bounded such that $\mathbb R^2 \backslash \overline {D}$ is connected. Assume $B \nsubseteq D$, then for any finite-dimensional subspace $V \subseteq \left[L^{2}\left(\mathbb S\right) \right] ^2 $ , there exists a sequence $\left\lbrace  \mathbf h_n \right\rbrace  \subseteq  V^{\perp} $ such that
 $$
 \left\|\mathbf H_B \mathbf h_n\right\|_{\left[H^{1/2}\left(\partial B\right) \right] ^2 } \to \infty \quad and \quad \left\| \mathbf G^* \mathbf h_n\right\|_{\left[H^{-1/2}\left(\partial D\right) \right] ^2 } \to 0, \qquad n\to \infty.
 $$
\end{thm}

\begin{proof}
Let $B, D  \subseteq \mathbb R^2$ is open and Lipschitz bounded such that $\mathbb R^2 \backslash \overline {D}$ is connected. Assume $B \nsubseteq D$, let $V \subseteq \left[L^{2}\left(\mathbb S\right) \right] ^2 $ is a finite-dimensional subspace, therefore, the orthogonal projection onto V is well-defined, and we denote it as $\mathbf P :\left[L^{2}\left(\mathbb S\right) \right] ^2 \to \left[L^{2}\left(\mathbb S\right) \right] ^2$. 

Since $B \nsubseteq D $ and $\mathbb R^2 \backslash \overline {D}$ is connected, there exists a relatively open $\tau \subseteq \partial B\backslash \overline{D}$ such that $\mathbb R^2 \backslash \overline {\left( \tau \cup D\right) }$ is connected, from Theorem \ref{RangeHG} we have
\begin{align}\label{eq:RangeHtG}
R\left( \mathbf H^*_{\tau} \right) \cap R\left( \mathbf G\right) =\left\lbrace 0\right\rbrace .
\end{align}

Now, we turn our attention to equation \eqref{eq:RestriH}. Without loss of generality, we assume that $\omega ^2$ is not an eigenvalue of $ \triangle ^*$ in $B$, therefore, both $\mathbf S _B$ (single-layer potential operator on $\partial B$ ) and $\mathbf G _B$ are injective. Moreover, the range of the extension operator $\mathbf R^*_{\tau} $ is infinite-dimensional, which implies that $ R \left( \mathbf H^*_{\tau} \right) = R \left( \sqrt{8\pi\omega}\mathbf{G_B S_B}  \mathbf R^*_{\tau}   \right) $ is infinite-dimensional. Therefore, according to lemma \ref{Dim} and \eqref{eq:RangeHtG} , we can obtain
$$
R \left( \mathbf H^*_{\tau} \right) \nsubseteq R\left( \mathbf G\right) +V=  R \left( \left[ \mathbf G \quad \mathbf P \right] \right).
$$
Hence, utilizing the Lemma \ref{Inquallity}, we can deduce that there does not exist a constant $C > 0$ such that
\begin{equation}\begin{aligned}\label{eq:inver}
|| \mathbf H_{\tau} \mathbf g ||^2_{\left[H^{1/2}\left(\tau \right) \right] ^2 }
&\leq C^2 \left\| 
\begin{bmatrix}
	\mathbf G^* \\ \mathbf P
\end{bmatrix}
\mathbf g \right\| ^2_{\left[H^{-1/2}\left(\partial D \right) \right] ^2  \times \left[L^{2}\left(\mathbb S \right) \right] ^2}\\
&=  C^2 \left( \left\| \mathbf G^* \mathbf g \right\|^2_{\left[H^{-1/2}\left(\partial D \right) \right] ^2}+\left\| \mathbf P \mathbf g \right\|^2_{\left[L^{2}\left(\mathbb S \right) \right] ^2} \right) .
\end{aligned}\end{equation}
Since $\mathbf P$ is an orthogonal projection operator, it follows that $\mathbf P$ is a self-adjoint operator, that is, $\mathbf P = \mathbf P ^*$. Therefore, from \eqref{eq:inver} for any $n \in \mathbb N$, there exists a $\mathbf g_n \in \left[L^{2}\left(\mathbb S \right) \right] ^2  $ such that
$$
\left\| \mathbf H_{\tau} \mathbf g_n \right\|^2_{\left[H^{1/2}\left(\tau \right) \right] ^2 } > n^2 \left( \left\| \mathbf G^* \mathbf g_n \right\|^2_{\left[H^{-1/2}\left(\partial D \right) \right] ^2}+\left\| \mathbf P \mathbf g_n \right\|^2_{\left[L^{2}\left(\mathbb S \right) \right] ^2} \right).
$$
For any $n \in \mathbb N$, we denote $M_n:=\left\| \mathbf G^* \mathbf g_n \right\|^2_{\left[H^{-1/2}\left(\partial D \right) \right] ^2}+\left\| \mathbf P \mathbf g_n \right\|^2_{\left[L^{2}\left(\mathbb S \right) \right] ^2}$,  $\tilde{\mathbf g_n}:=\mathbf g_n /\left(\sqrt{n M_n} \right) $. Then on one hand, we have
\begin{equation}\begin{aligned}
\left\| \mathbf H_{\tau} \tilde{\mathbf g_n} \right\|^2_{\left[H^{1/2}\left(\tau \right) \right] ^2 }
&= \frac{1}{n M_n} \left\| \mathbf H_{\tau} \mathbf g_n \right\|^2_{\left[H^{1/2}\left(\tau \right) \right] ^2 } \\
&>\frac{n}{M_n} \left( \left\|  \mathbf G^* \mathbf g_n \right\|^2_{\left[H^{-1/2}\left(\partial D \right) \right] ^2}+\left\|  \mathbf P \mathbf g_n \right\|^2_{\left[L^{2}\left(\mathbb S \right) \right] ^2} \right) \\
&= n.
\end{aligned}\end{equation}
i.e. 
$$
\left\| \mathbf H_{\tau} \tilde{\mathbf g_n} \right\|^2_{\left[H^{1/2}\left(\tau \right) \right] ^2 } > n \to \infty,\quad   n \to \infty .
$$
On the other hand, 
\begin{equation}\begin{aligned}
		\left\| \mathbf G^* \tilde{\mathbf g_n} \right\|^2_{\left[H^{-1/2}\left(\partial D \right) \right] ^2}+\left\| \mathbf P \tilde{\mathbf g_n} \right\|^2_{\left[L^{2}\left(\mathbb S \right) \right] ^2}
		&= \frac{1}{n M_n} \left(\left\| \mathbf G^* \mathbf g_n \right\|^2_{\left[H^{-1/2}\left(\partial D \right) \right] ^2}+\left\| \mathbf P \mathbf g_n \right\|^2_{\left[L^{2}\left(\mathbb S \right) \right] ^2} \right) \\
		&= \frac{1}{n}.
\end{aligned}\end{equation}
Thus, we can obtain
$$
\left\| \mathbf G^* \tilde{\mathbf g_n} \right\|^2_{\left[H^{-1/2}\left(\partial D \right) \right] ^2}+\left\| \mathbf P \tilde{\mathbf g_n} \right\|^2_{\left[L^{2}\left(\mathbb S \right) \right] ^2}= \frac{1}{n} \to 0, \quad n \to \infty .
$$
In summary, when $n \to \infty $, we can obtain
\begin{equation}\begin{aligned}\label{eq:HGPasymptotic}
  \left\| \mathbf H_{\tau} \tilde{\mathbf g_n} \right\|^2_{\left[H^{1/2}\left(\tau \right) \right] ^2 } \to \infty ,\quad \left\| \mathbf G^* \tilde{\mathbf g_n} \right\|^2_{\left[H^{-1/2}\left(\partial D \right) \right] ^2} \to 0,\quad \left\| \mathbf P \tilde{\mathbf g_n} \right\|^2_{\left[L^{2}\left(\mathbb S \right) \right] ^2} \to 0.
 \end{aligned}\end{equation}

For any $n\in \mathbb N$, we define $\mathbf h_n:= \tilde{\mathbf g_n} - \mathbf P \tilde{\mathbf g_n}$, using the triangle inequality, we can obtain
\begin{equation}\begin{aligned}\label{eq:HInequality}
	    \left\| \mathbf H_{\tau} \mathbf h_n \right\|_{\left[H^{1/2}\left(\tau \right) \right]^2 }
		&\geq \left| \left\| \mathbf H_{\tau} \tilde{\mathbf g_n} \right\|_{\left[H^{1/2}\left(\tau \right) \right]^2 } - \left\| \mathbf H_{\tau} \mathbf P \tilde{\mathbf g_n} \right\|_{\left[H^{1/2}\left(\tau \right) \right]^2 } \right |\\
		&\geq \left\| \mathbf H_{\tau} \tilde{\mathbf g_n} \right\|_{\left[H^{1/2}\left(\tau \right) \right]^2 } - \left\| \mathbf H_{\tau} \mathbf P \tilde{\mathbf g_n} \right\|_{\left[H^{1/2}\left(\tau \right) \right]^2 }\\
		&\geq \left\| \mathbf H_{\tau} \tilde{\mathbf g_n} \right\|_{\left[H^{1/2}\left(\tau \right) \right]^2 } - \left\| \mathbf H_{\tau} \right\|_{ \left[L^{2}\left(\mathbb S \right) \right]^2 \to \left[H^{1/2} \left(\tau \right) \right] ^2 } \left\|\mathbf P \tilde{\mathbf g_n} \right\|_{\left[L^{2} \left(\mathbb S \right) \right] ^2}.
\end{aligned}\end{equation}
and
\begin{equation}\begin{aligned}\label{eq:GInequality}
		\left\| \mathbf G^* \mathbf h_n \right\|_{\left[H^{-1/2}\left(\partial D \right) \right]^2 }
		&\leq \left\| \mathbf G^* \tilde{\mathbf g_n }\right\|_{\left[H^{-1/2}\left(\partial D \right) \right]^2 }+ \left\| \mathbf G^* \mathbf P \tilde{\mathbf g_n }\right\|_{\left[H^{-1/2}\left(\partial D \right) \right]^2 } \\
		&\leq \left\| \mathbf G^* \tilde{\mathbf g_n }\right\|_{\left[H^{-1/2}\left(\partial D \right) \right]^2 }+ \left\| \mathbf G^*\right\|_{\left[L^{2} \left(\mathbb S \right) \right] ^2 \to \left[H^{-1/2}\left(\partial D \right) \right]^2} \left\|\mathbf P \tilde{\mathbf g_n }\right\|_{\left[L^{2} \left(\mathbb S \right) \right] ^2}.
\end{aligned}\end{equation}
From \eqref{eq:HGPasymptotic} , \eqref{eq:HInequality} and  \eqref{eq:GInequality} we have $ \left\| \mathbf H_{\tau} \mathbf h_n \right\|_{\left[H^{1/2}\left(\tau \right) \right]^2 } \to \infty $, $\left\| \mathbf G^* \mathbf h_n \right\|_{\left[H^{-1/2}\left(\partial D \right) \right]^2 } \to 0$, when $n \to \infty$. From the definition of the restriction operator $\mathbf R_{\tau}$, we have that $\mathbf H_{\tau}= \mathbf R_{\tau} \mathbf H_{B} $ holds, so we can obtain
\begin{equation}\begin{aligned}
\left\| \mathbf H_{\tau} \mathbf h_n \right\|_{\left[H^{1/2}\left(\tau \right) \right]^2 } 
&=\left\|\mathbf R_{\tau} \mathbf H_{B} \mathbf h_n \right\|_{\left[H^{1/2}\left(\partial B \right) \right]^2 } \\
& \leq \left\|\mathbf R_{\tau}\right\|_{\left[H^{1/2}\left(\partial B \right) \right]^2 \to \left[H^{1/2}\left(\tau \right) \right]^2 } \left\| \mathbf H_{B} \mathbf h_n \right\|_{\left[H^{1/2}\left(\partial B \right) \right]^2 } ,
\end{aligned}\end{equation}
i.e. we have $ \left\| \mathbf H_{B} \mathbf h_n \right\|_{\left[H^{1/2}\left(\partial B \right) \right]^2 } \geq \left\| \mathbf H_{\tau} \mathbf h_n \right\|_{\left[H^{1/2}\left(\tau \right) \right]^2 }  / \left\|\mathbf R_{\tau}\right\|_{\left[H^{1/2}\left(\partial B \right) \right]^2 \to \left[H^{1/2}\left(\tau \right) \right]^2 }$, since when $n \to \infty $,  it follows that $\left\| \mathbf H_{\tau} \mathbf h_n \right\|_{\left[H^{1/2}\left(\tau \right) \right]^2 } \to \infty $ and $\left\|\mathbf R_{\tau}\right\|_{\left[H^{1/2}\left(\partial B \right) \right]^2 \to \left[H^{1/2}\left(\tau \right) \right]^2 }$ is a constant, we have
$$
\left\| \mathbf H_{B} \mathbf h_n \right\|_{\left[H^{1/2}\left(\partial B \right) \right]^2 } \to \infty ,\quad n\to\infty.
$$

Therefore, with the above expression and \eqref{eq:GInequality} , the proof is completed.
\end{proof}

\section{Characterization of the obstacle}\label{sec:Characterization of the obstacle}

In this section, we present the principal results of this work, namely the theorem on the characterization of the shape of the obstacle boundary, before formally introducing this theorem, we first need to introduce several notations and operators.

First, we define a boundary value mapping operator and demonstrate its compactness. Let $\overline{B}\subseteq D$, we define boundary value mapping operator $\mathbf M_{B \to D}$:$\left[H^{1/2}\left(\partial B \right) \right]^2 \to \left[H^{1/2}\left(\partial D \right) \right]^2$ as
\begin{equation}\label{eq:BVO}
		\mathbf M_{B \to D} \mathbf f := \mathbf u |_{\partial D},
\end{equation}
where $\mathbf u \in \left[H^{1}_{loc}\left(\mathbb R \backslash \overline{B} \right) \right]^2 $ is the unique solution to the following exterior Dirichlet boundary value problem
\begin{equation}\label{eq:ueq}
	\begin{cases}
		\triangle ^* \mathbf u + \omega ^2 \mathbf u =0, &\qquad \mbox{in}\quad \mathbb R^2 \backslash \overline{B},\\
		\mathbf u= \mathbf f , & \qquad\mbox{on} \quad \partial B ,\\
		\lim\limits_{\rho \to \infty} \rho^{\frac{1}{2}} \left( \frac{\partial \mathbf u_{\beta}}{\partial \rho} - ik_{\beta}\mathbf u_{\beta}\right) =0, &\quad \rho =|\mathbf x|, \beta = p,s.
	\end{cases}
\end{equation}

Furthermore, from the definition of the data-to-pattern operator, we can obtain
$$
	\mathbf G \mathbf M_{B \to D} \mathbf f = \mathbf w^{\infty}, 
$$
where $\mathbf w \in \left[H^{1}_{loc}\left(\mathbb R \backslash \overline{D} \right) \right]^2$ is the unique solution to the following exterior Dirichlet boundary value problem
\begin{equation}
	\begin{cases}
		\triangle ^* \mathbf w + \omega ^2 \mathbf w =0, &\qquad \mbox{in}\quad \mathbb R^2 \backslash \overline{D},\\
		\mathbf w= \mathbf M_{B \to D}\mathbf f , & \qquad\mbox{on} \quad \partial D ,\\
		\lim\limits_{\rho \to \infty} \rho^{\frac{1}{2}} \left( \frac{\partial \mathbf w_{\beta}}{\partial \rho} - ik_{\beta}\mathbf w_{\beta}\right) =0, &\quad \rho =|\mathbf x|, \beta = p,s.
	\end{cases}
\end{equation}
Since $\overline{B} \subseteq D$, it follows that $\mathbf u^{\infty} = \mathbf w^{\infty}$ and therefore we have $\mathbf G \mathbf M_{B \to D} \mathbf f = \mathbf w^{\infty} = \mathbf u^{\infty} =\mathbf G_B \mathbf f$, so we can obtain $\mathbf G_B=\mathbf G \mathbf M_{B \to D}$.

Next, we aim to demonstrate that $\mathbf M_{B \to D}$ is a compact operator. To prove this claim, we need the following lemma, and the proof of this lemma can be directly obtained from the Theorem 8.8 in \cite{GT01}.
\begin{lem}\label{ElliptRegular}
Let $\mathcal L u:= \partial _i \left( a^{ij} \left( x \right) \partial_j u +b^i \left( x \right) u\right) +c^i \left( x \right) \partial_i u+ d\left( x \right) u$ , $i,j=1,2$, and $\mathcal L$ be a coercive elliptic operator in the domain $\Omega$, and $u\in \mathbf H^1 \left( \Omega\right)$ be a weak solution to $\mathcal L u=f $. The coefficients $a^{ij}$ and $b^{i}$ of the differential operator are uniformly Lipschitz continuous in $\Omega$, the coefficients $c^{i}$, $d$ are essentially bounded in $\Omega$. Then for any subdomain $\Omega^{'} \subseteq \Omega$, we have $u\in \mathbf H^2 \left( \Omega^{'}\right) $.
\end{lem}
Using the preceding lemma, we now demonstrate through the following lemma that G is a compact operator.

\begin{lem}\label{BoundaryVCompac}
	Let $\overline{B}\subseteq D$, the boundary value mapping operator $ \mathbf M_{B \to D}$ is a compact operator from $ \left[H^{1/2}\left(\partial B \right) \right]^2 $ to $\left[H^{1/2}\left(\partial D \right) \right]^2$.
\end{lem}
\begin{proof}
 From equation \eqref{eq:ueq}, we know that $\mathbf u$ satisfies the Navier equation in $\mathbb R^2 \backslash \overline{B}$, note that the Navier equation's second-order operator is strictly elliptic, and its coefficients meet the requirements for $\mathcal{L}$ in Lemma \ref{ElliptRegular}. So for a subset $D^{'}\subseteq \mathbb R^2$, with $\overline{D}\subseteq D^{'} $, we can ascertain that $\mathbf u$ also satisfies the Navier equation in $D^{'}\backslash \overline{B}$, and since $\overline{B}\subseteq D$, there exists an $\Omega^{'}\subseteq D^{'} \backslash \overline{B}$ such that $\partial D\subseteq \Omega^{'}$. Therefore, through the Lemma \ref{ElliptRegular}, we can obtain  $\mathbf u\in \left[H^{2}\left(\Omega^{'} \right) \right]^2$. 
 
 Furthermore, by the trace theorem, we can obtain $\mathbf u |_{\partial D} \in \left[H^{3/2}\left(\partial D \right) \right]^2$, so the boundary value mapping operator can be rewritten as
 $$
 \mathbf M_{B \to D} = \mathbf E \mathbf M^{'}_{B \to D}.
 $$
Here, $\mathbf E$ is the compact embedding operator from $\left[H^{3/2}\left(\partial D \right) \right]^2$ to $\left[H^{1/2}\left(\partial D \right) \right]^2$, and $\mathbf M^{'}_{B \to D} := \left[H^{1/2}\left(\partial B \right) \right]^2 \to \left[H^{3/2}\left(\partial D \right) \right]^2$ is defined as
$$
\mathbf M^{'}_{B \to D} \mathbf f:= \mathbf u |_{\partial D}.
$$
Thus, it can be concluded that $ \mathbf M_{B \to D}=\mathbf E \mathbf M^{'}_{B \to D}$ is a compact operator from $\left[H^{1/2}\left(\partial B \right) \right]^2$ to $\left[H^{1/2}\left(\partial D \right) \right]^2$.

The proof is completed.
\end{proof}
Next, for two compact self-adjoint operators, we define a notation regarding the number of negative eigenvalues of their difference.
\begin{defn} \label{defleq} 
	Let $X$ be a Hilbert space, and $A_1,A_2: X \to X $ be compact self-adjoint linear operators, if $A_2-A_1$ has at most $r$ negative eigenvalues (where $r\in \mathbb N$), we define this as
	$$
	A_1 \leq_r A_2.
	$$
\end{defn}

In particular, when the number of negative eigenvalues of $A_2-A_1$ is finite (i.e., we can find $r \in \mathbb N$ such that $A_2-A_1$ satisfies $A_1 \leq_r A_2$), we denote this as $A_1 \leq_{fin} A_2$. Meanwhile, Below, we establish an equivalent condition under which the difference of the two operators has finitely many negative eigenvalues. Moreover, the following lemma is taking from \cite{HPS19b}.
\begin{lem}\label{EquDef}
	Let $X$ be a Hilbert space with the inner product defined as $\left\langle  \cdot , \cdot \right\rangle _X$, and let $r \in \mathbb N$. Then the following two statements are equivalent:
	\begin{itemize}
		\item [(1).] $ A_1 \leq_r A_2 $.
		\item [(2).] There exists a finite-dimensional subspace $V \subseteq X$ and $dim \left( V\right) \leq r$, such that:
		$$
		\left\langle  \left( A_2- A_1\right) v, v \right\rangle _X \geq 0, \quad \forall v \in V^{\perp}.
		$$
	\end{itemize}
\end{lem}

Next, we utilize the eigenvalue distribution properties of the operator $-\mathbf H^*_B\mathbf H_B-\Re\left(\mathbf F \right) $ to establish a criterion for determining the shape of the obstacle $D$. This decision rule is presented in the following theorem.
\begin{thm}\label{CharacterofObstacle}
	Let $B,D\subseteq \mathbb R^{2}$ be an open, Lipschitz bounded domain such that $\mathbb R^2\backslash D$ is connected,
	\begin{itemize}
		\item [(1).] If $\overline{B}\subseteq D$, then we have $\Re\left( \mathbf F\right) \leq_{fin} -\mathbf H^*_B\mathbf H_B$;
		\item [(2).] If $B\nsubseteq D$, then we have $\Re\left( \mathbf F\right) \nleq_{fin} -\mathbf H^*_B\mathbf H_B$.
	\end{itemize} 
\end{thm}
\begin{proof}
First, we proceed to prove the first part of the theorem.

In equation \eqref{eq:HGSoperator}, by replacing $D$ with $B$, we obtain the relation $\mathbf H^*_B=\sqrt{8\pi\omega}\mathbf G_B \mathbf S_B$. Furthermore, from the definition of the boundary value mapping operator $\mathbf M_{B \to D}$ in \eqref{eq:BVO}, we can obtain $ \mathbf H_B= \sqrt{8\pi\omega}\mathbf S^*_B \left( \mathbf M_{B \to D} \right) ^* \mathbf G^*$, therefore, based on the factorization \eqref{eq:FarfieldOperatorDecomposition} of the far-field operator in Lemma \ref{Fdecomposition}  we can obtain
$$
\Re\left( \mathbf F\right)+\mathbf H^*_B\mathbf H_B = -\sqrt{8\pi\omega}\mathbf G \left[ \frac{1}{2}\left( \mathbf S^*+\mathbf S\right) \right] \mathbf G^*+ 8\pi\omega\mathbf G \left[\mathbf M_{B \to D}\mathbf S_B\mathbf S^*_B\left( \mathbf M_{B \to D}\right) ^*\right] \mathbf G^*,
$$
where
\begin{equation}\begin{aligned}
\frac{1}{2}\left( \mathbf S^*+\mathbf S\right) 
&= \frac{1}{2}\left( 2\mathbf S_{i}+\mathbf S^*+\mathbf S-2\mathbf S_{i}\right)\\
&= \mathbf S_{i} + \frac{1}{2}\left( \left( \mathbf S^*-\mathbf S_{i}\right) +\left( \mathbf S-\mathbf S_{i}\right) \right) .
\end{aligned}\end{equation}
From Lemma \ref{SProperty}, we know that $\mathbf S-\mathbf S_i$ is a compact operator, so if we define $\mathbf K:= \mathbf S-\mathbf S_{i}$ , then $\mathbf K$ is a compact operator. Additionally, since the adjoint of a compact operator remains compact, it follows that $\mathbf S^*-\mathbf S_{i}=\left(\mathbf S-\mathbf S_{i}\right) ^*=\mathbf K^*$ still is a compact operator. Let $\mathbf K_1:=\mathbf K +\mathbf K^*$, then we obtain $\frac{1}{2}\left( \mathbf S^*+\mathbf S\right)=\mathbf S_{i} +\mathbf K_1$, meaning $\frac{1}{2}\left( \mathbf S^*+\mathbf S\right)$ is a compact perturbation of the self-adjoint and coercive operator $\mathbf S_{i}$. Thus, we can derive the following equation
$$
\Re\left( \mathbf F\right)+\mathbf H^*_B\mathbf H_B = -\mathbf G \left[\sqrt{8\pi\omega}\left( \mathbf S_{i} +\mathbf K_1 \right)-8\pi\omega \mathbf M_{B \to D}\mathbf S_B\mathbf S^*_B\left( \mathbf M_{B \to D}\right) ^*\right]\mathbf G^*.
$$
We denote $\mathbf K_2:=8\pi\omega \mathbf M_{B \to D}\mathbf S_B\mathbf S^*_B\left( \mathbf M_{B \to D}\right) ^*$, according to the compactness of the boundary value mapping operator $\mathbf M_{B \to D}$ demonstrated in Lemma \ref{BoundaryVCompac}, it follows that $\mathbf K_2$ is a compact operator from $ \left[H^{1/2}\left(\partial B \right) \right]^2 $ to $\left[H^{1/2}\left(\partial D \right) \right]^2$. Meanwhile, we write $\mathbf{\tilde{K}}=\sqrt{8\pi\omega} \mathbf K_1+\mathbf K_2$, from which we can obtain $\Re\left( \mathbf F\right)+\mathbf H^*_B\mathbf H_B = -\mathbf G \left( \sqrt{8\pi\omega} \mathbf S_{i}+\mathbf{\tilde{K}} \right) \mathbf G^*$, where $\mathbf{\tilde{K}}:\left[H^{1/2}\left(\partial B \right) \right]^2  \to \left[H^{1/2}\left(\partial D \right) \right]^2$ is a self-adjoint compact operator, therefore for any $\mathbf g \in \left[L^{2}\left(\mathbb S \right) \right]^2$, we have
\begin{equation}\begin{aligned}
		\left\langle \left( \Re\left( \mathbf F\right)+\mathbf H^*_B\mathbf H_B\right) \mathbf g, \mathbf g\right\rangle 
		&= \left\langle \left(-\mathbf G \left( \sqrt{8\pi\omega} \mathbf S_{i}+\mathbf{\tilde{K}} \right) \mathbf G^* \right) \mathbf g, \mathbf g\right\rangle  \\
		&= -\left(  \mathbf G^* \mathbf g, \sqrt{8\pi\omega} \mathbf S_{i}\mathbf G^* \mathbf g \right)  -\left(  \mathbf G^* \mathbf g, \mathbf{\tilde{K}}\mathbf G^* \mathbf g \right)  \\
		&\leq -c || \mathbf G^* \mathbf g ||^2_{\left[H^{-1/2}\left(\partial D \right) \right]^2} - \left(  \mathbf G^* \mathbf g, \mathbf{\tilde{K}}\mathbf G^* \mathbf g \right) ,
\end{aligned}\end{equation}
	where the last inequality is derived from the coerciveness of $\mathbf S_{i}$ from lemma \ref{SProperty}. For the second term on the right-hand side of the above equation, according to Lemma \ref{DatatoPatternProperty}, we know that there exists a finite-dimensional subspace $V \subseteq \left[L^{2}\left(\mathbb S \right) \right]^2$ such that
	$$
	 - \left(  \mathbf G^* \mathbf g, \mathbf{\tilde{K}}\mathbf G^* \mathbf g \right) \leq c ||\mathbf G ^* \mathbf \mathbf g||_{\left[H^{-1/2}\left(\partial D\right) \right] ^2} ^2, \qquad \forall \mathbf g \in V^\perp,
	$$
	Thus we can obtain 
	$$
	\left\langle \left( \Re\left( \mathbf F\right)+\mathbf H^*_B\mathbf H_B\right) \mathbf g, \mathbf g\right\rangle  \leq -c || \mathbf G^* \mathbf g ||^2_{\left[H^{-1/2}\left(\partial D \right) \right]^2} + c ||\mathbf G ^* \mathbf \mathbf g||_{\left[H^{-1/2}\left(\partial D\right) \right] ^2} ^2=0,\qquad \forall \mathbf g \in V^\perp.
	$$
	Therefore, according to Lemma \ref{EquDef}, it can be concluded that the first part of this theorem has been proved.
	
	Next, we employ a proof by contradiction to begin demonstrating the second part of this theorem.
	
	If $B\nsubseteq D$, suppose there exists a finite-dimensional subspace $V_1\in \left[L^{2}\left(\mathbb S \right) \right]^2$ such that
	$$
		\left\langle \left( \Re\left( \mathbf F\right)+\mathbf H^*_B\mathbf H_B\right) \mathbf g, \mathbf g\right\rangle \leq 0, \qquad \forall \mathbf g \in V_1^\perp
	.$$
	From the definitions of the inner product and norm in Hilbert space, we can deduce that
	\begin{equation}\label{eq:InequalityA}
		\left\langle \mathbf H^*_B\mathbf H_B \mathbf g ,\mathbf g\right\rangle =||\mathbf H_B \mathbf g||^2_{\left[H^{1/2}\left(\partial B \right) \right]^2}.
	\end{equation}
Meanwhile, since $\Re\left( \mathbf F\right) = -\sqrt{8\pi\omega}\mathbf G \left[ \frac{1}{2}\left( \mathbf S^* + \mathbf S\right) \right] \mathbf G^*$, we can infer that
\begin{equation}\begin{aligned}\label{eq:InequalityB}
-\left\langle \Re\left( \mathbf F\right) \mathbf g, \mathbf g\right\rangle
&\leq |\left\langle \Re\left( \mathbf F\right) \mathbf g, \mathbf g\right\rangle| \\
&= \left| \left( -\frac{\sqrt{8\pi\omega}}{2} \left( \mathbf S^*+\mathbf S\right) \mathbf G^* \mathbf g, \mathbf G^* \mathbf g \right) \right| \\
&\leq \frac{\sqrt{8\pi\omega}}{2} ||\mathbf S^*+\mathbf S ||_{\left[H^{-1/2}\left(\partial D\right) \right] ^2 \to \left[H^{1/2}\left(\partial D\right) \right] ^2} ||\mathbf G^* \mathbf g||^2_{\left[H^{-1/2}\left(\partial D \right) \right]^2}\\
&\leq c^{'}||\mathbf G^* \mathbf g||_{\left[H^{-1/2}\left(\partial D \right) \right]^2}.
\end{aligned}\end{equation}
It is straightforward to observe that the above $c^{'} > 0 $. Combining \eqref{eq:InequalityA} and \eqref{eq:InequalityB}, we can deduce that
	$$
	\left\langle \left( \Re\left( \mathbf F\right)+\mathbf H^*_B\mathbf H_B\right) \mathbf g, \mathbf g\right\rangle \geq - c^{'}||\mathbf G^* \mathbf g||^2_{\left[H^{-1/2}\left(\partial D \right) \right]^2} + ||\mathbf H_B \mathbf g||^2_{\left[H^{1/2}\left(\partial B \right) \right]^2},
	$$
	For the right-hand side of the above inequality, according to Theorem \ref{LocialWaveFounc}, we know that there exists a sequence $\left\lbrace \mathbf g_n \right\rbrace \subseteq V_1^\perp$ such that $||\mathbf G^* \mathbf g||_{\left[H^{-1/2}\left(\partial D \right) \right]^2} \to 0$ and $||\mathbf H_B \mathbf g||^2_{\left[H^{1/2}\left(\partial B \right) \right]^2} \to \infty$ as $n\to \infty$, so we can obtain that there exists $\mathbf g \in V_1^\perp $ such that $ \left\langle \left( \Re\left( \mathbf F\right)+\mathbf H^*_B\mathbf H_B\right) \mathbf g, \mathbf g\right\rangle > 0$, this contradicts our assumption, thus completing the proof of the second part of the theorem.
\end{proof}

\section{Conclusion and future works}

For the inverse elastic wave scattering problem with Dirichlet boundary conditions, we have proposed a monotonicity-based shape characterization method to reconstruct the shape of rigid obstacle scatterers. In future work, we will attempt to extend this method to inverse elastic wave scattering problems with Neumann boundary conditions. Furthermore, we will consider shape reconstruction of obstacles under mixed boundary conditions and validate the effectiveness of our approach through numerical examples.



\end{document}